\documentclass[10pt,a4paper,reqno]{amsart} 

\usepackage{amsmath}
\usepackage{mathtools}
\usepackage{amssymb}
\usepackage{graphicx}
\usepackage{color}
\usepackage{latexsym}
\usepackage[utf8]{inputenc}
\usepackage[T1]{fontenc}
\usepackage{url}
\usepackage{comment}

\newcommand{\qs}{{\mathbb Q}}  
\newcommand{\cs}{{\mathbb C}} 
\newcommand{\rs}{{\mathbb R}} 

   \renewcommand\Im{\operatorname{Im}}

\newcommand{\la}{\lambda}

\newcommand{\bs}{\bar s}
\newcommand{\bS}{\bar S}

\newcommand{\by}{\bar y}

\newcommand{\E}{\mathbb{E}}

\DeclareMathOperator{\Var}{Var}

\newcommand{\cC}{\mathcal C}
\newcommand{\cD}{\mathcal D}

\newcommand{\cW}{\mathcal W}

\newcommand{\cP}{\mathcal P}

\newcommand{\cT}{\mathcal T}

\newtheorem{Theorem}{Theorem}
\newtheorem{Proposition}[Theorem]{Proposition}

\newtheorem{Lemma}[Theorem]{Lemma}

\newcommand{\beq}{\begin{equation}}
\newcommand{\eeq}{\end{equation}}

\newcommand{\gf}{generating function}
\newcommand{\gfs}{generating functions}

\newcommand{\saws}{self-avoiding walks}
\def\emm#1,{{\em #1}}

 \def\NN{{\sf N}}
 \def\EE{{\sf E}}
 \def\SS{{\sf S}}
 \def\WW{{\sf W}}

\graphicspath{{Figures/}}

\catcode`\@=11
\def\section{\@startsection{section}{1}%
 \z@{.7\linespacing\@plus\linespacing}{.5\linespacing}%
 {\normalfont\bfseries\scshape\centering}}

\def\subsection{\@startsection{subsection}{2}%
  \z@{.5\linespacing\@plus\linespacing}{.5\linespacing}%
  {\normalfont\bfseries\scshape}}

\def\subsubsection{\@startsection{subsubsection}{3}%
 \z@{.5\linespacing\@plus\linespacing}{-.5em}
  {\normalfont\bfseries\itshape}}
\catcode`\@=12

\def\qed{$\hfill{\vrule height 3pt width 5pt depth 2pt}$}

\begin{document}
\title
[On the importance sampling of self-avoiding walks]
{On the importance sampling of self-avoiding walks}

\author[M. Bousquet-M\'elou]{Mireille Bousquet-M\'elou}
%
\address{MBM: CNRS, LaBRI, Universit\'e Bordeaux 1, 
351 cours de la Lib\'eration, 33405 Talence, France}
\email{mireille.bousquet@labri.fr}

\thanks{}


\begin{abstract}
In a 1976 paper published in \emm Science,, Knuth presented an
algorithm to sample  (non-uniform) self-avoiding walks crossing a
square of side $k$. From this sample, he constructed an estimator for the
number of such walks. The quality of this estimator is directly
related to the (relative) variance of a certain random variable $X_k$. From his
experiments, Knuth suspected that this variance was extremely large (so that
the estimator would not be very efficient). But how large?
For the analogous Rosenbluth algorithm, which samples \emm unconfined,
self-avoiding walks of length $n$,  the variance of the corresponding
estimator is believed to be exponential in $n$.

A few years ago, Bassetti and Diaconis showed that, for a 
sampler \emm \`a la Knuth, that  generates walks 
crossing a $k\times k$ square and consisting of
North and East steps,
the relative variance is only $O(\sqrt k)$. In this note we take one step
further and show that, for  walks consisting of North, South and
East steps, the relative variance jumps to
$2^{k(k+1)}/(k+1)^{2k}$.
This is quasi-exponential in the average length of the
 walks, which is of order $k^2$. We also obtain partial
results for  general self-avoiding walks crossing a square, suggesting 
that the relative variance could  be exponential in $k^2$ (which is again  the average length of these walks).

Knuth's algorithm  is a basic example of  a widely used technique
called \emm sequential importance
sampling,.  The present paper, following Bassetti and Diaconis' paper,
is one of very few 
examples where the variance of the estimator can be found.
\end{abstract}

\date{August 24th, 2012}
\maketitle


\section{Introduction}
A \emm self-avoiding walk, (SAW) on a graph is a walk that never visits
 the same vertex twice.  Let $\cW_k$ be the set of SAWs on a $k \times
k$ square grid, going from the South-West vertex  
to the North-East vertex (Figure~\ref{fig:10by10-SAW}). In his paper ``\emm Coping with
finiteness,''~\cite{knuth,knuth-selected}, Knuth 
described the following algorithm to generate a (non-uniform) random
walk of $\cW_k$: start from the South-West
corner, and at each time, choose with equal probability (which can be $1/3$, $1/2$ or 1) one of the
\emm eligible, steps. A step is \emm eligible, if, once appended to
the current walk, it gives a self-avoiding walk 
that can be extended so as to end at the North-East corner.
In this way the walk is never trapped and
the algorithm always succeeds\footnote{We describe in  Section~\ref{sec:trap}
  how to detect algorithmically when   a new step traps the walk.}.
Figure~\ref{fig:2by2-SAW}
shows the probabilities of the 12 possible walks when $k=2$. Two bigger
examples ($k=10$, $k=100$) are shown in Figure~\ref{fig:10by10-SAW}.
 This procedure is a basic example of  a widely used technique called \emm sequential importance
sampling,~\cite{bassetti-diaconis,chen,diaconis-graphs}. It is also a
variant, for walks confined to a square, of the Rosenbluth algorithm
that generates unconfined SAWs~\cite{rosenbluth}.

\begin{figure}[htb]
\includegraphics[scale=1]{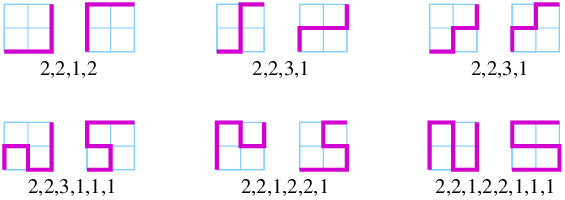}
\caption{The  12 self-avoiding walks crossing the $2\times 2$
  square. For each of them, we give the 
  sequence $1/p_1, 1/p_2, \ldots$ where $p_i$ is the probability of
  the $i$th step.  The probability of the walk is thus the reciprocal of the
  product of the terms in the list. Two walks have probability $1/8$,
  six have probability $1/12$, and four have probability $1/16$. Two
  walks that differ by a diagonal symmetry have the same probability.}
\label{fig:2by2-SAW}
\end{figure}

\begin{figure}[b]
\includegraphics[scale=0.4]{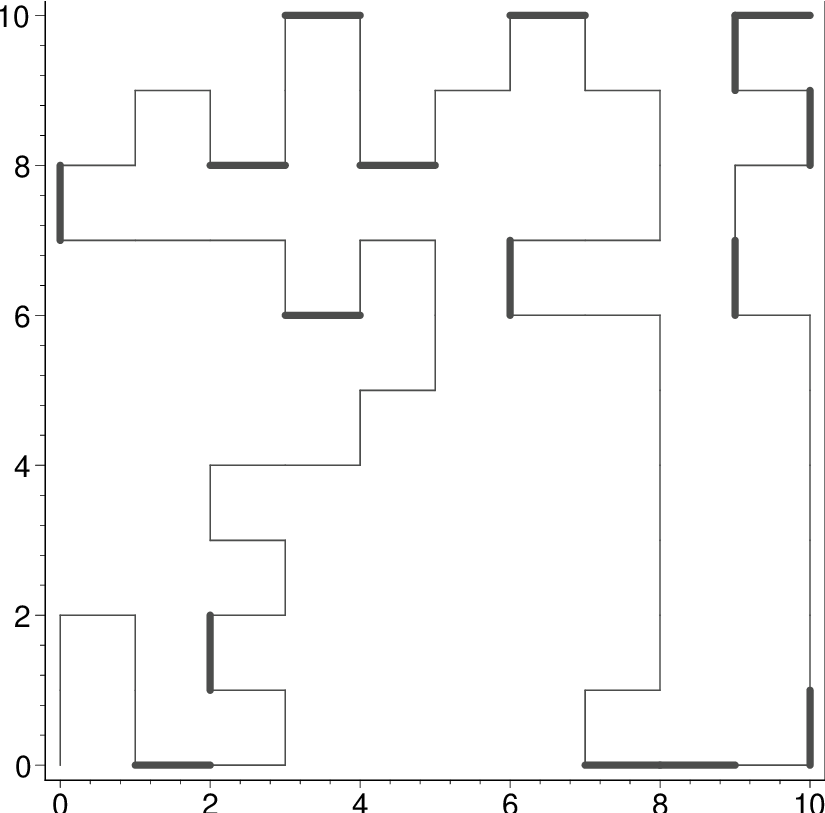}
\hskip 20mm \includegraphics[scale=0.4]{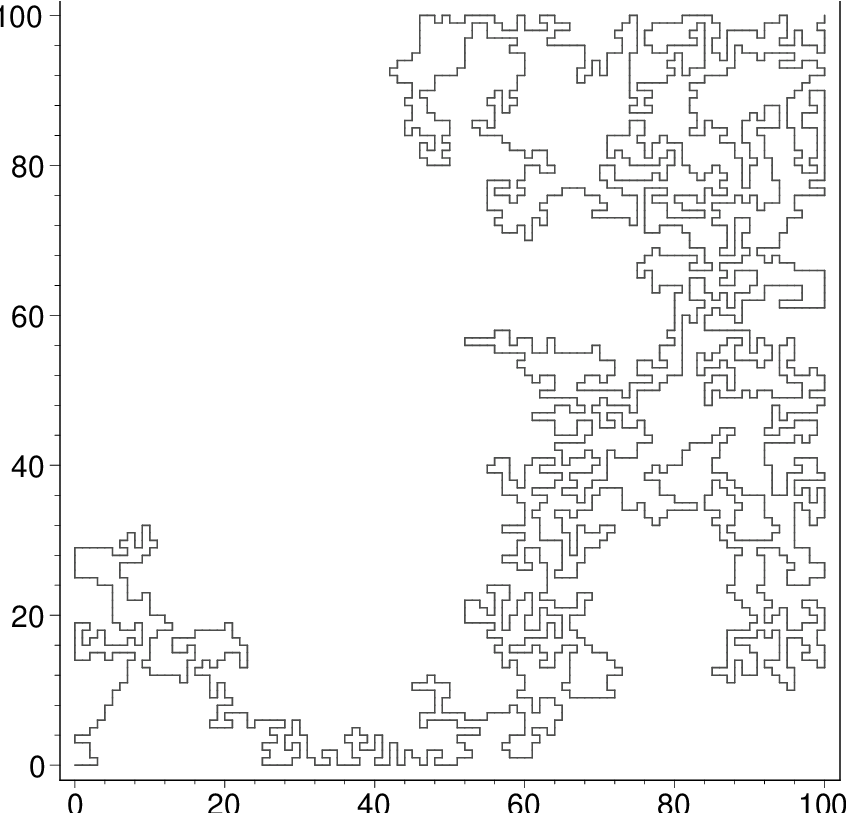}
\caption{Left: A SAW crossing the $10\times 10$
  square. The thick steps have probability 1. That is, each of them is the
  only eligible step at the time when it is taken. Right: A SAW crossing the
  $100\times 100$   square, obtained via Knuth's algorithm.}
\label{fig:10by10-SAW}
\end{figure}

Denote by $p(w)$ be the probability to draw the walk $w \in
\cW_k$. Consider the random variable $X_{k}=1/p(w)$, where $w$ is a
random walk of 
$\cW_{k}$ drawn according to the distribution $p(\cdot)$. Clearly,
$$
\E(X_{k})= \sum_{w \in  \cW_{k}} p(w) \frac 1 {p(w)} = |\cW_{k}|.
$$
Hence one can estimate the number of SAWs crossing a $k\times k$ square
by generating $N$ walks $w^{(1)}, \ldots, w^{(N)}$ of $\cW_k$, and
computing
\beq\label{eq:estimator}
\frac 1 N \sum_{i=1}^N \frac 1 {p(w^{(i)})}.
\eeq
By generating ``several thousand'' walks for $k=10$, Knuth obtained 
$$
|\cW_{10}|\simeq (1.6\pm 0.3) \times 10^{24},
$$
which is quite good compared to the now known exact value,
$1,568,758,030,464,750,013,214,100$ (see~\cite{mbm-crossing,knuth-selected}).
We have reproduced Knuth's experiment,  and found, with a first group of
10,000 walks, the estimate $1.78  \times 10^{24}$, and with a second
group, the  estimate $1.38  \times 10^{24}$.
As observed by Knuth, the
values $\frac 1 {p(w^{(i)})}$  vary a lot (a small sample of 10 walks gave us values ranging from $10^{11}$ to
$10^{24}$),  and one may suspect that the variance of $X_{k}$
is probably  much larger than $\E(X_k)^2$, or, in other words, that
the \emm relative variance,
$$
\Var\left( \frac {X_k}{\E(X_k)}\right)
$$
is large.
Note that 
$$\Var(X_{k})=\E(X_{k}^2)-\E(X_{k})^2= \sum_{w \in  \cW_{k}}
\frac 1 {p(w)} - |\cW_k|^2.  
$$
Also, observe that the variance of the estimator~\eqref{eq:estimator} is
$\Var(X_k)/N$, so that $\Var(X_k)$ is a measure of the quality of this
estimator.  Let us mention that, even though sequential
importance sampling is widely used, no general bounds on the variance
of the estimators are available. 


\medskip
Knuth's observation led Bassetti and Diaconis to study  a simpler
algorithm, in which a step, to be eligible, has to go North (\NN) or
East (\EE)~\cite{bassetti-diaconis}. The resulting walk is called a \emm directed, walk, or 
\NN\EE-walk. Each step has probability $1/2$ unless it follows the
North or East side of the square --- in which case it has
probability~1. Denote by $\cD_k$ the set of directed walks of $\cW_k$,
and by $p(w)$ the probability to generate the directed walk
$w$ with this new algorithm. Define the random variable $X_k=1/p(w)$ as above. Then
$$
\E(X_k)= |\cD_k|={2k \choose k} \sim \frac{4^k}{\sqrt {\pi k}}.
$$
Of course, since $|\cD_k|$ is known exactly, there is no point in
using importance sampling to estimate this cardinality --- but it is
interesting to know that the
variance of the estimator can be determined, as follows.  By the above argument,  a walk of $\cD_k$ that hits the North
or East side of the square  for the first time at time $k+i$ has probability
$1/2^{k+i}$. Since there are $2 {k+i-1 \choose i}$ such walks,
$$
\E(X_k^2)= \sum_{w \in  \cD_{k}}\frac 1 {p(w)} =\sum_{i=0}^{k-1}2^{k+i+1} {k+i-1 \choose i}.
$$
The corresponding \gf\ is
\beq\label{gf-directed}
\sum_{k\ge 1}\E(X_k^2)x^k= \frac{2\,x}{1+2x} \left( {\frac {3}{\sqrt {1-16\,x}}}-1 \right)  ,
\eeq
and an elementary singularity analysis~\cite[Chap.~VI]{flajolet-sedgewick} gives 
$$
\E(X_k^2)\sim \frac{16^k}{3\sqrt{\pi k}},
$$
which is roughly $\sqrt k$ times larger than
$$\E(X_k)^2\sim \frac{16^k}{ {\pi k}}.$$


\medskip
In this note, we first take one more step in the direction of the general problem
by declaring that South steps are also eligible. The resulting walks are
 \emm partially directed, walks, or 
\NN\EE\SS-walks.  The probabilities of the 9 walks obtained when $k=2$
are shown in Figure~\ref{fig:2by2}. Of course, these probabilities are
not the same as those obtained from Knuth's original algorithm.

\begin{figure}[htb]
\scalebox{1}{\input{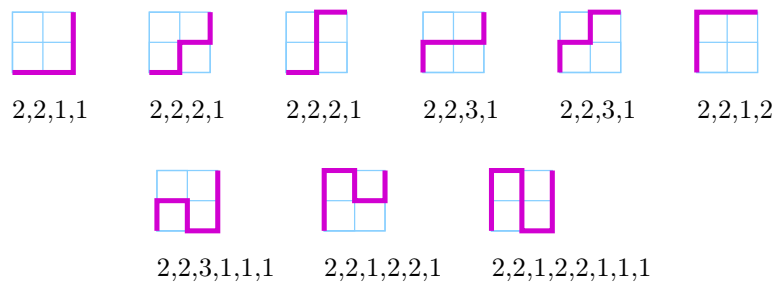}}
\caption{The  nine \NN\EE\SS-walks crossing a $2\times 2$ square, with the reciprocals of the
  probabilities of their steps.}
\label{fig:2by2}
\end{figure}

We will prove that one outcome of this increased generality is that the
ratio between $\E(X_k^2)$ and $\E(X_k)^2$ becomes much larger:
$$
\E(X_k^2) \sim \frac 3 2\, 2^{k(k+1)}  \quad \hbox{while} \quad \E(X_k)^2=(k+1)^{2k}.
$$
We will also prove that the average length of a (uniform) \NN\EE\SS-walk
confined to the $k\times k$ square is quadratic in $k$, so that the
variance of $X_k$ is roughly exponential in the  length, as predicted
for SAWs generated by the Rosenbluth algorithm~\cite{batoulis-kremer}.

Since the $x/y$ symmetry
is lost with partially directed walks, 
it is natural to
generalize the original question by 
enclosing  walks in a rectangle $R$ of height $k$ and width
$\ell$. 
Thus, let $\cP_{k,\ell}$ be 
the set of partially directed  
walks that start from the South-West corner of $R$ and end
at the North-East corner.  A walk of
$\cP_{k,\ell}$ contains exactly $\ell$ East steps, and choosing the
heights of these steps determines the walk completely. Hence the
number of walks in 
$\cP_{k,\ell}$ is $(k+1)^\ell$. 

Thus,  defining the random variable $X_{k,\ell}=1/p(w)$ as above, we have
\beq\label{EXkl}
\E(X_{k,\ell})=
|\cP_{k,\ell}|=(k+1)^\ell.
\eeq
We will prove that, if $k, \ell \rightarrow \infty$ in
such a way $\ell=o(2^k)$, then  
$$
\Var(X_{k,\ell}) \sim \E(X_{k,\ell}^2) \sim \frac 3 2\, 2^{(k+1)\ell},
$$
so that the relative variance satisfies
$$
\Var\left(\frac{X_{k,\ell}}{\E(X_{k,\ell})}\right) \sim \frac 3 2\,
\left(\frac{2^{k+1}}{(k+1)^2}\right)^\ell.
$$
This results thus extends the very short list of examples where the variance of an
importance sampler can be rigorously
established~\cite{bassetti-diaconis}.  Moreover, the average length of
a (uniform) walk of $\cP_{k,\ell}$ is shown to be approximately
$k\ell/3$, so   that the variance is again exponential in the length,
as expected for other similar samplers.

\medskip
The paper is organized as follows. In Section~\ref{sec:exact}, we
obtain  an explicit expression for the \gf\ of the 
numbers $\E(X_{k,\ell}^2)$ (the counterpart
of~\eqref{gf-directed}). In Section~\ref{sec:asympt}, we derive from
this expression the above asymptotic result. 
In Section~\ref{sec:knuth}, we go back to Knuth's sampler and prove that there exist two positive 
constants $\lambda$ and $\beta$ such that 
$$
\E(X_k)^{1/k^2} \rightarrow \lambda \quad \hbox{and} \quad \E(X_k^2)^{1/k^2} \rightarrow \beta.
$$ 
The former result has actually been known since 1978~\cite{abbott}.
Since a variance is non-negative,  $\beta  \ge
\lambda^2$. Upper (and lower) bounds on $\lambda$ have been obtained
in~\cite{mbm-crossing}, based on the determination of the numbers
$\E(X_k)=|\cW_k|$ for small values of $k$, and of related numbers
counting other configurations of \saws.  As shown in
Section~\ref{sec:knuth}, a similar study,
performed for the numbers $\E(X_k^2)$, might suffice to  prove that $\beta  >
\lambda^2$, so that the variance would be again exponential in $k^2$ (which
is known to be the average length of a uniform SAW crossing the $k\times
k$-square~\cite{madras}). 
We conclude with a few remarks and questions on the
importance sampling of self-avoiding walks not confined to a box.

\section{Exact results for \NN\EE\SS-walks}\label{sec:exact}
In this section, we first describe the probability $p(w)$ to obtain
the walk $w$ in terms
of the geometry of $w$ (Section~\ref{sec:proba}). This
description reduces the determination of the numbers $\E(X_{k,\ell}^2)$
to the enumeration of  \NN\EE\SS-walks  according to several
parameters, which we perform in Section~\ref{sec:enum}. 
\subsection{The probability $p(w)$}\label{sec:proba}
Let $w_0$ be a walk of $\cP_{k,\ell}$, written as a sequence of \NN,
\EE\ and \SS\ steps. Let $w$ be the prefix of $w_0$ that precedes the
last \EE\ step. That is, $w_0= w \EE \NN\cdots\NN$. By convention,
$w_0$ starts at height 0.
\begin{Lemma}\label{lem:proba}
The probability $p(w_0)$ to obtain $w_0$ via the importance sampling
algorithm satisfies
$$
\frac 1 {p(w_0)}= 2 \cdot 3^{h(w)} 2^{h_c(w)} 2^{v(w)} 1^{v_c(w)},
$$
where
\begin{itemize}
\item $h(w)$ is the number of horizontal steps of $w$ that lie neither
  at height $0$ nor at height $k$,
\item  $h_c(w)$ is the number of horizontal \emm contacts, of $w$,
  that is, horizontal steps  that lie   at height $0$ or $k$,
\item $v(w)$ is the number of vertical steps of $w$ that end neither
  at height $0$ nor at height $k$,
\item  $v_c(w)$ is the number of vertical \emm contacts, of $w$,
  that is, vertical steps  that end   at height $0$ or $k$.
\end{itemize}
\end{Lemma}
\begin{proof}
  Assume the walk $w_0$ has length $n$ and ends with exactly $j$ vertical steps. The probability of the first step is $1/2$, and the probability of
 each of  the $j$ final steps is 1. Let $s_i$ denote the $i$th
 step. Hence $w$ consists of the steps $s_1, \ldots, s_{n-j-1}$. For
  $1\le i<n-j$, the probability of $s_{i+1}$ depends on the direction and
  position of $s_{i}$:
  \begin{itemize}
  \item if $s_i$ is horizontal, but not a contact, then the
    probability of  $s_{i+1}$ is $1/3$,
 \item if $s_i$ is a horizontal contact, then the
    probability of  $s_{i+1}$ is $1/2$,
\item if $s_i$ is vertical, but not a contact, then the
    probability of  $s_{i+1}$ is $1/2$,
\item if $s_i$ is a vertical contact, then the
    probability of  $s_{i+1}$ is $1$.
  \end{itemize}
The lemma follows.
\end{proof}

\subsection{Enumeration of \NN\EE\SS-walks in a strip of
  fixed height}
\label{sec:enum}

Recall the expression~\eqref{EXkl} of the numbers $\E(X_{k,\ell})$. 
 For $k$ (the height of the rectangle) fixed,  the \gf\  of the numbers 
$\E(X_{k,\ell})^2$ 
is  \emm rational,:
\beq\label{rat-expectation}
\sum_{\ell \ge 1} \E(X_{k,\ell})^2 x^\ell = \sum_{\ell \ge
  1}(k+1)^{2\ell} x^\ell= \frac{(k+1)^2x}{1-(k+1)^2x}.
\eeq

We will determine  the variance of $X_{k,\ell}$ by describing the
\gf\ of the numbers $\E(X_{k,\ell}^2)$, which is also rational when $k$ is fixed.

\begin{Proposition}\label{prop:rat}
  For any fixed height $k$, the \gf\ $M_k(x)$ of the numbers $\E(X_{k,\ell}^2)$
  is a rational series:
$$
M_k(x):=\sum_{\ell \ge 1} \E(X_{k,\ell}^2) x^\ell = 2x \,\frac{N_k}{G_k},
$$
where 
$N_k$ and $G_k$ are polynomials in $x$ satisfying the same recurrence relation:
$$
N_{k}= (5+9x )N_{k-2} -4N_{k-4},
$$
(and similarly for $G_k$), with initial conditions 
$$
\begin{array}{lllllllll}
N_{1}&=&2,&& G_1&=&1-4x,\\
N_2&=&  5+3x,&&G_2&=&1 -9x-6x^2,\\
N_3&=& 11+9x,&&G_3&=& 1-19x-18x^2,\\
N_4&=& 23+54x+27x^2,&\hskip 8mm &G_4&=& 1-36x-99x^2-54x^3.
\end{array}
$$
\end{Proposition}

\smallskip
\noindent{\bf Example.} For $k=2$, 
$$
\sum_{\ell \ge 1} \E(X_{2,\ell}^2) x^\ell = 2x \,\frac{5+3x}{1-9x-6x^2}=
10x+96x^2+O(x^3).
$$
Figure~\ref{fig:2by2} allows us  to check that the coefficient of $x^2$ is correct:
$$
96=4+8+ 8 + 12 +12 + 8+ 12+ 16 +16.
$$
 The \gf\ of the variances is
$$
\sum_{\ell \ge 1} \Var(X_{2,\ell}) x^\ell = 2x
\,\frac{5+3x}{1-9x-6x^2}-\frac {9x} {1-9x}.
$$
Observe that the radius of the first fraction is smaller than the radius
of the second fraction. As $\ell\rightarrow \infty$,
$$
\E(X_{2,\ell}^2)\sim  \mu^\ell
$$
(up to a multiplicative constant) with $\mu= (9+\sqrt{105})/2\simeq
9.62$, while $\E(X_{2,\ell})^2= 9^\ell$. \qed

\medskip
We now want to prove Proposition~\ref{prop:rat}. Recall that
$$
\E(X_{k,\ell}^2)=\sum_{w_0\in \cP_{k,\ell}} \frac 1 {p(w_0)}.
$$
The expression of $p(w_0)$ given in  Lemma~\ref{lem:proba} leads us to
study a purely  enumerative problem.
For $k$ fixed, let $\cT_k$ be the set  of \NN\EE\SS-walks $w$ that start at
height $0$ and are confined to the strip $0\le y \le k$. We wish to
count these walks by the parameters $h(w)$, $h_c(w)$, $v(w)$ and
$v_c(w)$. So, let 
$$
\sum_{w\in \cT_k} x^{h(w)} y^{v(w)} a^{h_c(w)} b^{v_c(w)}
$$
be the associated \gf. This series is easily seen to be rational
(it can be determined using a transfer-matrix
approach~\cite[Sec.~4.7]{stanley-vol1}, or equivalently a finite-state
automaton~\cite[p.~362]{flajolet-sedgewick}), and  
there are several ways to determine it. We present here what we believe
to be the most direct one. It relies on a recursive description of the
walks of $\cT_k$, where we add at each time an \EE\ step and a sequence
of vertical steps. This approach requires to take into account an
additional parameter, namely the height $f(w)$ of the final point of the
walk $w$. Hence our series finally involve 5 variables:
$$
T_k(s)\equiv T_k(x,y,a,b;s)=\sum_{w\in \cT_k} x^{h(w)} y^{v(w)}
a^{h_c(w)} b^{v_c(w)}s^{f(w)}. 
$$
We will  denote by $\tilde \cT_k$ the subset of $\cT_k$ formed of
walks that do not  end at height 0 or $k$, and by $\tilde T_k(s)\equiv
\tilde T_k(x,y,a,b;s)$ the corresponding \gf.
Accordingly,
$$
T_k(s)=\sum_{i=0}^k T_{k,i} s^i= T_{k,0}+\tilde T_k (s) + s^k T_{k,k},
$$
where $T_{k,i}$ is the series in $x$, $y$, $a$ and $b$ counting walks
of $\cT_k$ ending at height $i$.
By Lemma~\ref{lem:proba}, 
\begin{eqnarray}
M_k(x)=  \sum_{\ell\ge 1} \E(X_{k,\ell}^2)x^\ell &=&\sum_{\ell\ge 1} \sum_{w_0\in
  \cP_{k,\ell} } \frac 1{p(w_0)}x^\ell\nonumber\\
&=&2\sum_{w\in \cT_k}3^{h(w)} 2^{h_c(w)}2^{v(w)} x^{1+h(w)+h_c(w)}\nonumber
\\
&=&2x\,T_k(3x,2,2x,1;1).\label{prob-enum}
\end{eqnarray}
{\bf Remark.} The series $T_{k,k}$ has already been determined
  in the case $x=y=a=b=t$,
 using the same approach as here~\cite[Prop~3]{bacher-mbm}. The derivation
is more involved here because we keep track of four parameters in the
enumeration, and because we are interested in $T_k(1)$ rather than $T_{k,k}$.
\begin{Lemma}\label{lem:rec}
  The series $\tilde T_k(s)$, $T_{k,0}$ and $T_{k,k}$ satisfy the following system of equations:
  \begin{multline*}
    \left(1-\frac x{1-ys} - \frac {xy\bs}{1-y\bs}\right) \tilde T_k(s)\\
= \frac{ys-(ys)^k}{1-ys} -\frac {x (ys)^k}{1-ys} \tilde T_k(1/y)
- \frac x{1-y\bs} \tilde T_k(y)+ aT_{k,0}  \frac{ys-(ys)^k}{1-ys}
+aT_{k,k} \frac{ys^{k-1} -y^k}{1-y\bs},
  \end{multline*}
  \begin{eqnarray*}
    T_{k,0}&=& 1+bx\by \tilde T_k(y)+aT_{k,0}+aby^{k-1}T_{k,k},
\\
T_{k,k}&=& by^{k-1} + bxy^{k-1} \tilde T_k(1/y)+ a by^{k-1}T_{k,0} + aT_{k,k},
  \end{eqnarray*}
with $\bs=1/s$ and $\by=1/y$.
\end{Lemma}
\begin{proof}
  We construct the walks of $\tilde \cT_k$ recursively, by adding at each time
  a horizontal step followed by a sequence of vertical steps.  

We partition the set $\tilde \cT_k$ into three disjoint subsets, illustrated in
Figure~\ref{fig:decomp}.
\begin{itemize}
\item 
The first subset  consists of walks with no \EE\ step. These walks
 consist of $i$ North steps, with $1\le i <k$. 
 Their \gf\ is  
$$
\sum_{i=1}^{k-1} (ys)^i= \frac{ys-(ys)^k}{1-ys} .
$$
  \item 
The second subset consists of
walks in which the last \EE\ is followed by a (possibly empty)
sequence of \NN\ steps.  We denote by   $i$ the height of the last
\EE\ step, and distinguish the cases $i=0$ and $0< i <k$. The \gf\ of
this subset of $\tilde \cT_k$ reads  
\begin{align*}
 aT_{k,0} \sum_{j=1}^{k-1} (sy)^j
+ x\sum_{i=1}^{k-1}\left(  T_{k,i} s^i \sum_{j=0}^{k-i-1} (ys)^j \right)
=
 aT_{k,0}  \frac{ys-(ys)^k}{1-ys}
+ x\sum_{i=1}^{k-1}\left(  T_{k,i} s^i \frac{1-(ys)^{k-i}}{1-ys}\right)\\
=
aT_{k,0}  \frac{ys-(ys)^k}{1-ys}
+\frac x{1-ys} \left( \tilde T_k(s)-(ys)^k \tilde T_k(1/y)\right).
\end{align*}
   \item 
The third subset  consists of
walks in which the last \EE\  step is followed by a non-empty sequence
of \SS\ steps. We  denote by $i$ the height of the last
\EE\ step, and distinguish the cases $i=k$  and $0<i<k$.  The \gf\ of
this subset  of $\tilde \cT_k$ reads
\begin{align*}
 a s^k T_{k,k} \sum_{j=1}^{k-1} (y\bs)^j+
x\sum_{i=1}^{k-1}\left(  T_{k,i} s^i \sum_{j=1}^{i-1}(y\bs)^j\right)
=
 a T_{k,k} \frac{ys^{k-1} -y^k}{1-y\bs}
+x\sum_{i=1}^{k-1}\left(  T_{k,i} s^i \frac{y\bs-(y\bs)^i}{1-y\bs}\right)\\
= a T_{k,k} \frac{ys^{k-1} -y^k}{1-y\bs}
+
\frac x{1-y\bs} \left(y\bs \tilde T_k(s)-\tilde T_k(y) \right).
\end{align*}
 \end{itemize}
Adding the three contributions gives the series $\tilde T_k(s)$ and
establishes the first equation of the lemma.

The equations for $T_{k,0}$ and $T_{k,k}$ are obtained in a similar fashion.
\end{proof}

\begin{figure}[htb]
\scalebox{0.8}{\input{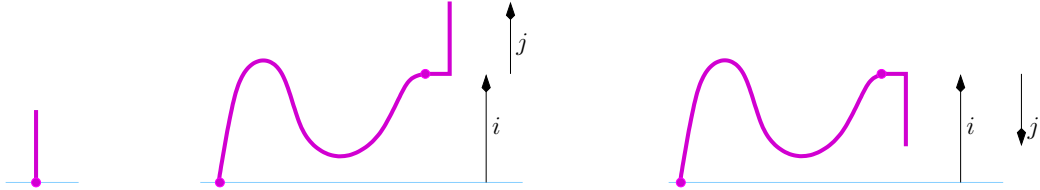}}
\caption{Recursive construction of bounded \NN\EE\SS-walks.}
\label{fig:decomp}
\end{figure}

We now solve the functional equations of Lemma~\ref{lem:rec}. The key
tool is the  \emm kernel method, (see
e.g.~\cite{hexacephale,bousquet-petkovsek-1,prodinger}).
\begin{Proposition}\label{prop:sol-enum}
Let $k\ge 1$. The series $T_k(1)\equiv T_k(x,y,a,b;1)$ counting
\NN\EE\SS-walks confined to a strip of height $k$ is 
$$
T_k(1)= \frac{N_k}{G_k},
$$
where $N_k$ and $G_k$ are polynomials in $x$, $y$, $a$ and $b$ satisfying
 the same recurrence relation
$$
N_{k}= (1-x+y^2 (1+x))N_{k-2} -y^2N_{k-4},
$$
(and similarly for $G_k$) with initial conditions
$$
\begin{array}{llllllll}
N_{-1}&=&(1-x-xy)(b-y)/y^2,
&\hskip 5mm \ &G_0&=& ( x -1) ab/y- ( x+1 )(a-1),
\\
N_0&=& (b-xb+xy)/y, &&G_1&=&1-a-ab,
\\
N_1&=& 1+b, &&G_2&=&(1-x)(1-a)- ( x+1 ) yab,
\\
N_2&=& 1-x+y+by(1+x), &&G_3&=&(1-x-xy)(1-a)-yab ( x+y+xy).
\end{array}
$$
Equivalently,
\beq\label{Tk1-sol}
T_k(1)
=
\frac 1 {P(\bS)S^k+P(S)} \left(Q(\bS)S^k+Q(S)
+(1-y^2)\frac{S^k-S}{S-1} \right),
\eeq
where $S$ is the unique formal power series\footnote{The other
  solution is $1/S$, and its expansion in $x$ and $y$ involves
  negative powers of $y$.} in $x$ and $y$ satisfying
\beq \label{ker}
S+\frac 1 S = (1+x)y +(1-x)\by,
\eeq
with $\by=1/y$, $\bS=1/S$, 
\beq\label{PQ-def}
P(s)= 1-a+aby-s(ab+y-ay) \quad\hbox{and } \quad  Q(s)=1-by+(b-y)s.
\eeq
\end{Proposition}
The reason why we give the expressions of $N_{-1}$ and $N_0$,
rather than $N_3$ and $N_4$, is that they are more compact. The same
reason explains why we give $G_0$ rather than $G_4$. It is of course easy
to compute  $N_3$, $N_4$ and $G_4$, and Proposition~\ref{prop:rat} then
follows at once, using~\eqref{prob-enum}.  We hope that using the same
notation $N_k$, $G_k$, for the enumeration problem of
Proposition~\ref{prop:sol-enum} and its
specialization of Proposition~\ref{prop:rat} will not create any confusion.

\begin{proof}
  First, we use the last two equations of Lemma~\ref{lem:rec} to express
  $\tilde T_k(y)$ and  $\tilde T_k(1/y)$ as linear combinations  of $T_{k,0}$
  and $T_{k,k}$.  Then, in the first equation of the lemma, we replace
  $\tilde T_k(y)$ and  $\tilde T_k(1/y)$ by their expressions in terms of  $T_{k,0}$
  and $T_{k,k}$. The  left-hand side is unchanged, and the right-hand
  side now involves only  two unknown 
  series, namely $T_{k,0}$ and $T_{k,k}$:
  \begin{multline}
    \left( 1-{\frac {x  }{1-ys}-{\frac {xy\bs}{1-y\bs}}} \right)\tilde T_k(s)
\\
=
 \frac {ys}{ 1-ys }
+\frac {y}{b \left( 1-y\bs \right) }
+ \left({\frac {ays}{1-ys}}+ 
{\frac {y \left( a -1\right) }{b \left( 1-y\bs \right) }} \right) 
T_{k,0}
+{s}^{k} \left( {\frac {y \left( a-1 \right) }{ b\left( 1-ys \right) 
}}+{\frac {ya\bs}{1-y\bs}} \right) T_{k,k}.
\label{eq-func-2}
  \end{multline}
The \emm
kernel, of this equation is the coefficient of $\tilde T_k(s)$.
It vanishes when $s=S$ and $s=\bS:=1/S$, where $S$ is defined in the
proposition.  Since $\tilde T_k(s)$ is a polynomial in $s$, and $S$
and $\bS$ are Laurent series in $x$ and $y$ with finitely many
monomials with negative exponents, the series
$\tilde T_k(S)$ and
$\tilde T_k(\bS)$ are well-defined. Replacing $s$ by $S$ or $\bS$ in the
above
equation cancels the left-hand side, and hence the right-hand side. One
thus obtains two linear equations between $T_{k,0}$ and $T_{k,k}$, which
involve the series $S$. Solving them gives  expressions of $T_{k,0}$
and $T_{k,k}$ in terms of $S$  (the expression of $T_{k,k}$ is given
in~\eqref{Tkk-S} below).
By setting $s=1$ in~\eqref{eq-func-2}, one then expresses $\tilde
T_k(1)$ in terms of $S$, and finally $T_k(1)= T_{k,0}+T_{k,k}+\tilde
T_k(1)$. This gives~\eqref{Tk1-sol}.

\medskip
Observe that the expression~\eqref{Tk1-sol} is unchanged if we replace $S$ by
$\bS=1/S$. In particular, it can be written as a symmetric rational
function in $S$ and $\bS$ (with coefficients in $\qs(a,b,y)$). Since
$S$ and $\bS$ are the two roots 
of~\eqref{ker}, their symmetric functions are rational functions of
$x$ and $y$. This
implies that $T_k(1)$ is a rational series in $x$, $y$, $a$ and
$b$.  However,  the  
denominator of~\eqref{Tk1-sol}, namely $P(\bS)S^k+P(S)$, is not unchanged
when $S\mapsto 1/S$. But let us define 
the series   $N_k$ and $G_k$ as follows:
\begin{eqnarray}
  G_{2k}&=&\frac{y^k}{1-y^2}\left( P(\bS) S^k +P(S)\bS^k\right),\nonumber
\\
G_{2k+1}&=& \frac{y^k}{(1-y)(1+S)}\left( P(\bS) S^{k+1}
  +P(S)\bS^k\right), \label{G-N-S}
\\
N_{2k}&=&\frac{y^k}{1-y^2}\left( Q(\bS) S^k +Q(S)\bS^k+(1-y^2)
  \frac{S^{k}-\bS^{k-1}}{S-1}\right),\nonumber
\\
N_{2k+1}&=&\frac{y^k}{(1-y)(1+S)}\left( Q(\bS) S^{k+1} +Q(S)\bS^k+(1-y^2) \frac{S^{k}-\bS^k}{1-\bS}\right).\nonumber
\end{eqnarray}
Then  it is easy to check that~\eqref{Tk1-sol} can be rewritten as
$T_k(1)= N_k/G_k$. Moreover,  the series $N_k$ and $G_k$ are unchanged
when  $S\mapsto 1/S$, and hence, by the same argument as above, they are
rational functions of $x$, $y$, $a$ and $b$.  More precisely, each of the
sequences $G_{2k}$, $G_{2k+1}$, $N_{2k}$ and $N_{2k+1}$ is of the form
$y^k(\alpha S^k+\beta \bS^k)$, where $S$ and $\bS$ are the two roots
of~\eqref{ker}. Hence each  sequence satisfies the
recurrence relation
$$
u_{k}= (1-x+y^2 (1+x))u_{k-1} -y^2u_{k-2}.
$$
One easily determines the initial values for each sequence. This
yields the description of $N_k$ and $G_k$ given in the
proposition. From this description, it is clear that $N_k$ and $G_k$
are polynomials, as soon as $k\ge 1$.
 \end{proof}

\noindent
{\bf Remarks} \\
{\bf 1. Denominators.} The series $T_{k,i}$, counting walks ending at height $i$, are also rational, but with a denominator that is a proper multiple of the denominator of
$T_k(1)= \sum_{i=0}^k T_{k,i}$. For instance,
\beq\label{Tkk-S}
T_{k,k}= \frac{b(y^2-1)(\bS-S)S^k}{(P(\bS)S^k-P(S))(P(\bS)S^k+P(S))},
\eeq
or, in terms of polynomials,
\beq\label{Tkk-pol}
 T_{k,k}= \frac{by^{k-1}}{F_k}, 
\eeq
where $F_k$ is defined by
 the  recurrence relation
$$
  F_{k}= (1-x+(1+x)y^2 )F_{k-1} -y^2F_{k-2},
$$
with the initial conditions 
$$
F_1=(1-a-ab)(1-a+ab)
\quad \hbox{and} \quad 
F_2= (1-a+bya)((1-x)(1-a)- ( x+1 ) yab).
$$
It is not hard to prove that $G_k$, the denominator of $T_k(1)$, is a
divisor of $F_k$. The 
simplification that occurs in the denominator when summing the series
$T_{k,i}$ over $i$ has  recently been explained combinatorially, for
slightly different walk models, by Bacher~\cite{bacher-sym}. 

\smallskip
\noindent
{\bf 2. Average length.}
The series $T_{k,k}(tx,t,tx,t)$ counts \NN\EE\SS-walks crossing a strip of height $k$
 according to the length (variable $t$) and the width (variable $x$). In
particular, the case  $t=1$ of~\eqref{Tkk-pol} reads
$T_{k,k}(x,1,x,1)=1/(1-(k+1)x)$, as justified
combinatorially in the introduction. In order to determine the average
length $|w|$ of a uniform  \NN\EE\SS-walk $w$ crossing a $k\times \ell$ 
rectangle, we differentiate
$T_{k,k}(tx,t,tx,t)$ with respect
to $t$, and then set $t=1$. This gives
$$
\left. \frac{\partial}{\partial t}\left(T_{k,k}(tx,t,tx,t)\right)\right| _{t=1}
= \sum_{\ell \ge 0} x^\ell \sum_{w\in \cP_{k,\ell}}|w| 
= {\frac {k}{1-x \left( 1+k \right) }}
+{\frac {x \left( 1+k \right) 
 \left( 1+k \left( k+2 \right) x/3 \right) }{ \left( 1-x \left( 1
+k \right)  \right) ^{2}}},
$$
so that the average length is
\beq\label{av-le}
{\frac { \left( {k}^{2}+5\,k+3 \right) \ell}{3(1+k)}}+{\frac {k
 \left( 2\,k+1 \right) }{3(1+k)}}= \frac{k\ell}{3}+O(k+\ell).
\eeq

\section{Asymptotic results for \NN\EE\SS-walks}\label{sec:asympt}
We now derive  asymptotic results from the previous section. Recall
that $\E(X_{k,\ell})= (k+1)^\ell$, so that the \gf\ of the numbers
$\E(X_{k,\ell})^2$ is given by~\eqref{rat-expectation} (for $k$
fixed), with radius of convergence $1/(k+1)^2$. The radius of convergence of 
the \gf\ of the numbers $\E(X_{k,\ell}^2)$
turns out to be exponentially smaller. Our study has analogies with
the study of the longest run in a binary
string~\cite[p.~308]{flajolet-sedgewick}, which also requires 
to analyze (explicit) rational functions depending on an integer $k$.
\begin{Proposition}\label{prop:asympt}
 Let $k\ge 1$. The series $M_k(x)$, given in
 Proposition~{\rm\ref{prop:rat}}, has a unique pole $\rho_k$ of  modulus
 less than $1/9$, satisfying 
\beq\label{rho-exp}
\rho_k = \frac 1{2^{k+1}}+\frac 9{2 \cdot 4^{k+1}}
-\frac{12k-23}{2\cdot 8^{k+1}} 
+\frac{36k^2-54k-87/8}{16^{k+1}}+O\left(\frac{k^3}{32^k}\right).
\eeq
As $x\rightarrow \rho_k^-$,
\beq\label{M-exp}
M_k(x) \sim \frac{\alpha_k}{1-x/\rho_k}
\eeq
with
\beq\label{residue-exp}
\alpha_k=
\frac 3 2-\frac{9 k-4}{2^{k+2}}+\frac{27k^2-48k+1}{2\cdot 4^{k+1}}
-\frac{81k^3-306k^2+75k+140}{2\cdot 8^{k+1}}+O\left(\frac{k^4}{16^k}\right).
\eeq
The second moment of $X_{k,\ell}$  satisfies, uniformly in $k$ and $\ell$,
$$
\E(X_{k,\ell}^2) = \alpha_k \rho_k^{-\ell} + O(9^\ell k).
$$
In particular, if $k, \ell \rightarrow \infty$ in
such a way $\ell=o(2^k)$, then  
$$
\Var(X_{k,\ell}) \sim \E(X_{k,\ell}^2) \sim \frac 3 2\, 2^{(k+1)\ell},
$$
which is much larger than
$\E(X_{k,\ell})^2=(k+1)^{2\ell}$. By~\eqref{av-le}, the
variance is thus exponential in the average length of a (uniform) \NN\EE\SS-walk
crossing the $k\times \ell$ rectangle.   
\end{Proposition}
\begin{proof}
We proceed in four steps. We first express the series $M_k(x)$ in
terms of an algebraic series $S$, as was done for the enumerative
problem in Proposition~\ref{prop:sol-enum}.  Then, we study the analytic properties
of $S$. We use these properties to prove that the denominator $G_k$ of
$M_k$ is real-rooted, with one positive zero $\rho_k$ and all the
other zeroes below $-1/9$. We finally apply Cauchy's formula to
extract the $\ell$th coefficient of $M_k(x)$, which is
$\E(X_{k,\ell}^2)$.

\medskip
\noindent{\bf Step 1. The expression of $M_k$}
\\
By Proposition~\ref{prop:rat},
$$
M_k(x)=2x \frac{N_k}{G_k},
$$
where the polynomials $N_k$ and $G_k$ can be described either by
induction, or, after performing the change of variables  $x\rightarrow 3x$, $a\rightarrow 2x$,
$y\rightarrow 2$ and $b\rightarrow 1$ in~\eqref{G-N-S}, by\footnote{From now on, we
  carefully avoid the notation $\bS:=1/S$, since we will soon be doing
  complex analysis.}
\begin{eqnarray}
  G_{2k}&=&-\frac{2^k}{3}\left( P(1/S) S^k +P(S)S^{-k}\right),\nonumber
\\
G_{2k+1}&=& -\frac{2^k}{1+S}\left( P(1/S) S^{k+1}
  +P(S)S^{-k}\right), \label{G-N-S-simple}
\\
N_{2k}&=&-\frac{2^k}{3}\left( Q(1/S) S^k +Q(S)S^{-k}-3
  \frac{S^{k}-S^{-k+1}}{S-1}\right),\nonumber
\\
N_{2k+1}&=&-\frac{2^k}{1+S}\left( Q(1/S) S^{k+1}
  +Q(S)S^{-k}-3\frac{S^{k}-S^{-k}}{1-1/S}\right),\nonumber
\end{eqnarray}
where $S$ and $1/S$ are the two power series in $x$ satisfying
\beq\label{ker-s}
S+\frac 1 S= \frac{5+9x}2,
\eeq
or equivalently,
\beq\label{xS}
x=-\frac 1 9 (2S-1)\left( 2/ S-1\right).
\eeq
The polynomials $P(s)$ and $Q(s)$ are 
$$
P(s) = 1+2x-2s(1-x) \quad \hbox{and} \quad
Q(s)=-1-s,
$$
so that, in view of~\eqref{xS},
$$
P(S)= \frac{(2S-1)(2S^2-11S-4)}{9S} \quad \hbox{and} \quad 
P(1/S)=\frac{(2-S)(2-11S-4S^2)}{9S^2}.
$$
It also follows from~\eqref{prob-enum} and~\eqref{Tk1-sol} that
\beq\label{MS}
M_k(x)=\frac {2x} {P(\bS)S^k+P(S)} \left(Q(\bS)S^k+Q(S)
-3\,\frac{S^k-S}{S-1} \right).
\eeq

\medskip
\noindent{\bf Step 2. The series $S(x)$}\\
From now on, we denote by $S$ the root of~\eqref{ker-s} that has
constant term 1/2:
\beq\label{S-def}
S= \frac{5+9x - 3\sqrt{(1+x)(1+9x)}}4.
\eeq
\begin{Lemma}\label{lem:S}
  The series $S$ has radius of convergence $1/9$, and admits an
  analytic continuation, still denoted by $S$, in
  $\cs\setminus[-1,-1/9]$. In this domain, $S$ never vanishes, and its
  modulus is less than $1$. 
\end{Lemma}
\begin{proof}
  The existence of an analytic continuation  follows from  basic complex
analysis. If $x=u+iv$, the imaginary part of the discriminant $(1+x)(1+9x)$ reads
$2v(5+9u)$. Using the principal determination of the square root, the
analytic continuation of  $S$ is given by~\eqref{S-def} when $\Re(x) \ge -5/9$, and
otherwise by
$$
S=\frac{5+9x +3\sqrt{(1+x)(1+9x)}}4.
$$
 A plot of the modulus of $S$ is shown in Figure~\ref{fig:S} (left).
\begin{figure}[htb]
\includegraphics [scale=0.3]{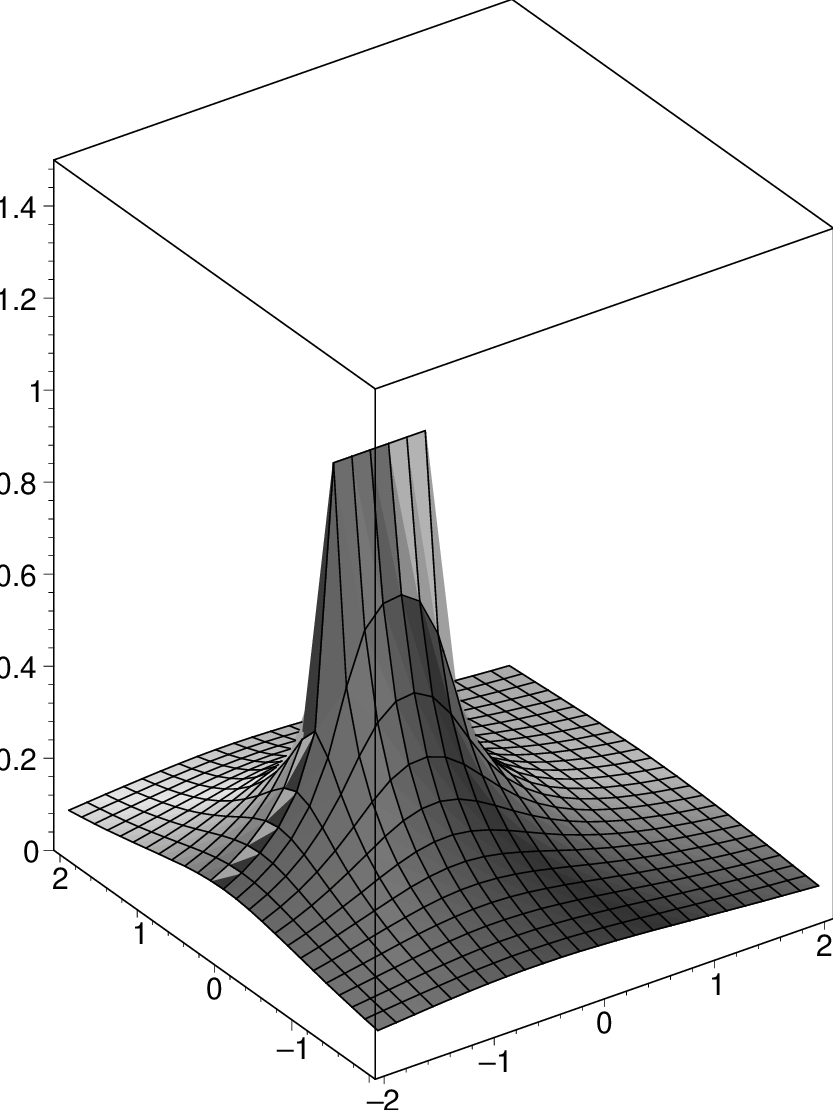}
\hskip 30mm
\includegraphics [scale=0.3]{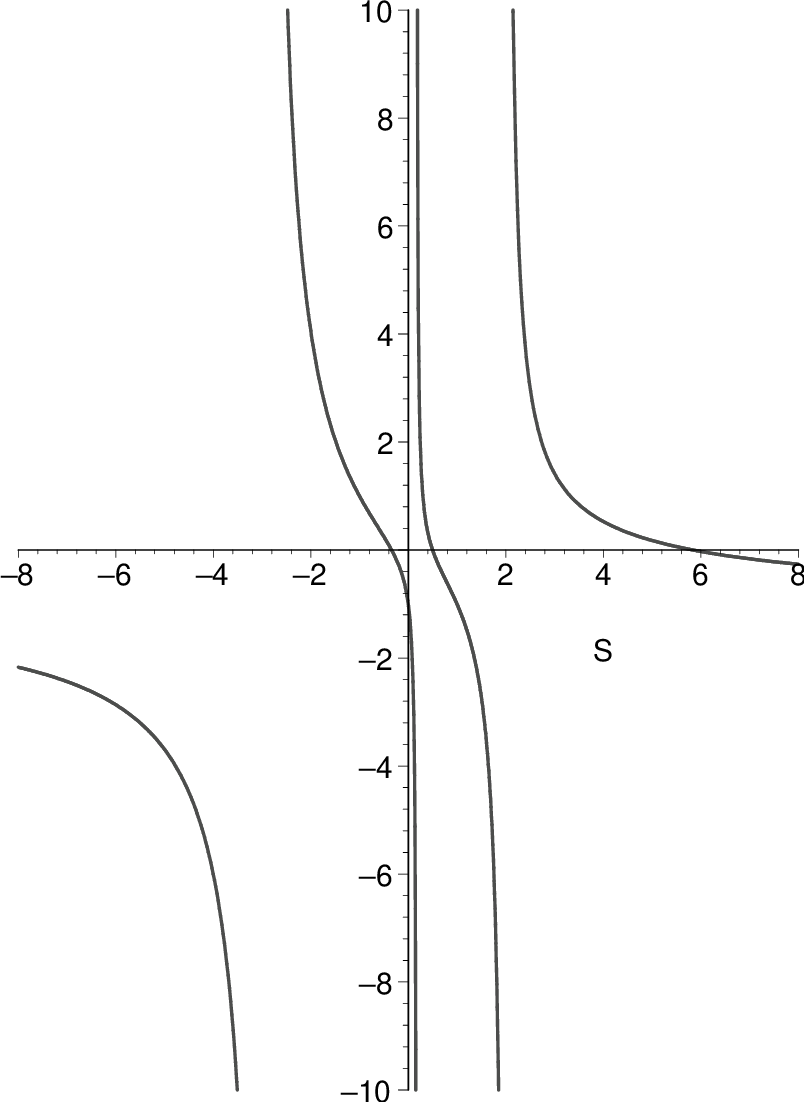}
\caption{Left: A plot of the modulus of  $S$, showing the  cut on the
  interval $[-1,-1/9]$. Right: The function $R(s)$.}
\label{fig:S}
\end{figure}
\end{proof}

\medskip
\noindent {\bf Step 3. The roots of $G_k$}
\begin{Lemma}\label{lem:zeroes}
  For $k\ge 1$, the denominator $G_k$ of the series $M_k$ is
  real-rooted. It has
 a unique positive zero $\rho_k$, which, as $k\rightarrow \infty$, admits the
  expansion~\eqref{rho-exp}. The other zeroes are smaller than $-1/9$.
As $x$ approaches $\rho_k$, the series $M_k$  behaves likes
$\alpha_k/(1-x/\rho_k)$, where $\alpha_k$ admits the
expansion~\eqref{residue-exp}. 
\end{Lemma}
We could use Rouch\'e's theorem to
    prove that, for any $\varepsilon >0$, the polynomial $G_k$ has only
    one root of modulus less than $1/9-\varepsilon$ for $k$ large
    enough, but the above statement is more precise.
\begin{proof}
 The case $k=1$ being trivial ($G_1=1-4x$), we focus on the case $k\ge
2$.  By Proposition~\ref{prop:rat}, the denominator $G_k$ has degree $\lceil
  \frac{k+1}2\rceil$. The expressions~\eqref{G-N-S-simple} are
  symmetric in $S$ and $1/S$, and thus hold for any determination of
  $S$, and thus for any $x \in \cs$, including in the cut
  $[-1,-1/9]$. They show that $G_k(x)=0$ if and only if $S\not = -1$ and
$$
S^{k-1}= 
-\frac{(2S-1)(2S^2-11S-4)}{(2-S)(2-11S-4S^2)}.
$$
Conversely, if $s \in \cs\setminus\{-1\}$ is a root of 
\beq\label{s-eq}
s^{k-1}= 
-\frac{(2s-1)(2s^2-11s-4)}{(2-s)(2-11s-4s^2)}:=R(s),
\eeq
then 
\beq\label{x-s}
x:=-\frac 1 9 (2s-1)\left( 2 /{s}-1\right)
\eeq
is a root of $G_k$. Observe that in this case, $1/s$ is also a root
of~\eqref{s-eq}, and gives rise to the same root $x$ of
$G_k$. Conversely, if two distinct
roots $s_0$ and $s_1$ of~\eqref{s-eq} give rise to the same root  of
$G_k$, then $s_1=1/s_0$. 

It is easy to relate the positions of $s$ and $x$ in the complex
plane. By writing $s=u+iv$, one finds that $x$ is real if and only if
$s$ is real or has modulus~1. If $s=e^{i\theta}$, then
$x=(4 \cos \theta-5)/9$ lies in $[-1,-1/9]$.   If $s$ is real and
negative,  then $x\le -1$, and the equality holds if and only if
$s=-1$. If $s$ is real and positive, then $x\ge -1/9$, and $x>0$ if
and only if $s\not \in [1/2,2]$.  

\medskip
Since we want to prove that $G_k$ is real-rooted,  let us study
the roots of~\eqref{s-eq}, distinct from $-1$, that are real or have
modulus 1.
We will prove that~\eqref{s-eq} has 
\begin{itemize}
\item  two pairs $\{s,1/s\}$ of  real zeroes  distinct from
$-1,$ one positive outside of  $[1/2,2]$, and one negative, 
\item   $\lceil \frac{k-3} 2\rceil$ pairs of zeroes distinct from $-1$
  on the unit circle. 
\end{itemize}
Consequently, $G_k$ has two real zeroes outside the interval
$[-1,-1/9]$, one positive, one less than $-1$, and $\lceil \frac{k-3}
2\rceil$ zeroes in $[-1,-1/9]$. In particular, it is real rooted.

\medskip
\noindent {\bf Real roots of~\eqref{s-eq}. } An elementary study of the function
$R(s)$, for $s \in \rs$, reveals that it consists of 4 decreasing
branches, shown in Figure~\ref{fig:S} (right), with vertical asymptotes at 
$$
s=-\frac{11+3\sqrt {17}}8\simeq -2.9,
\quad 
s= \frac{-11+3\sqrt {17}}8\simeq 0.17,
\quad \hbox{and} \quad s=2.
$$
The branches intersect the $s$-axis at the reciprocals of these three values (and in particular at $1/2$).
Thus in $\rs^+$, the equation $s^{k-1}=R(s)$ has two roots, one below
$1/2$ and the other beyond $2$, which are
necessarily the reciprocal of each other. The smallest of these
increases to $1/2$ as $k$ increases: thus the corresponding value of
$x$ decreases to $0$ as $k$ increases. We denote by $\rho_k$ this root
of $G_k$. 

If $k$ is even,   the equation $s^{k-1}=R(s)$
has also  two roots in $\rs^-$. If $k\ge 3$ is odd, the curve
$s\mapsto s^{k-1}$
intersects the second branch of $R(s)$ at $s=-1$, but also somewhere between $s=0$ and
$s=-0.508\ldots$ (which is the root obtained for $k=3$).
The latter intersection point  gives rise to a root of $G_k$ smaller than $-1$.

 The rest of the argument will show that all other roots of $G_k$ lie
 in $[-1,-1/9]$.

\medskip
\noindent {\bf  Roots of~\eqref{s-eq} of modulus 1. } We first observe
that, if $s$ has modulus 1, then the same holds for $R(s)$. More
precisely, if $s=e^{i\theta}$, then $R(s)=e^{i\phi}$ with
\begin{eqnarray*}
  \cos \phi &=& -{\frac {56+321\,\cos \theta -336\,  \cos ^{2} \theta +128
\,  \cos  ^{3}\theta }{ \left(5- 4\,\cos \theta \right) 
\left( 157 +44\,\cos \theta-32\,  \cos ^{2} \theta \right) }},
\\
\sin \phi &=&-27\,{\frac { \left(29- 16\,\cos \theta \right) \sin
 \theta}{ \left(5- 4\,\cos \theta
 \right)  
\left( 157 +44\,\cos \theta-32\,  \cos^{2} \theta  \right) }}.
\end{eqnarray*}
Plots of $\cos \phi$ and $\sin \phi$ as a function of $\theta$ are
shown in Figure~\ref{fig:R}. For $s=e^{i\theta}$, Eq.~\eqref{s-eq} is
equivalent to $\cos ((k-1)\theta)= \cos \phi$ and  $\sin
((k-1)\theta)= \sin \phi$. Given that $1/s= e^{-i\theta}$, we can
focus on solutions such that $\theta\in [0,\pi]$. The oscillations of
$\cos ((k-1)\theta)$ in this interval imply that the equation $\cos ((k-1)\theta)=
\cos \phi$ admits at least one solution in each interval $(
  \frac{m-1}{k-1} \pi, \frac{m}{k-1} \pi]$, for $1\le m \le k-1$. For
    each solution, $\sin((k-1)\theta)=\pm \sin \phi$, and the plot of
    $\sin \phi$ in Figure~\ref{fig:R} shows that $\sin((k-1)\theta)=
    \sin \phi$ if and only if $\sin((k-1)\theta)\le 0$, that is, if $m$
    is even. We finally note that, when $k$ is odd, one solution is
    $\theta=\pi$, giving $s=-1$, which we want to exclude. 

This discussion shows that~\eqref{s-eq} has at least $\lceil
\frac{k-3}2\rceil $ solutions $s\not
= -1$ with $\Im(s) > 0$ on the unit circle.  They give rise to as
many roots of $G_k$ in the interval $[-1,-1/9]$. With the two real
roots of $G_k$ found previously outside this interval, this gives a total of $\lceil
\frac{k+1}2\rceil $ roots, which coincides with the degree of
$G_k$. Hence $G_k$ is real rooted, with one positive root $\rho_k$, and
the others smaller than $-1/9$.

\begin{figure}[htb]
\includegraphics [scale=0.3]{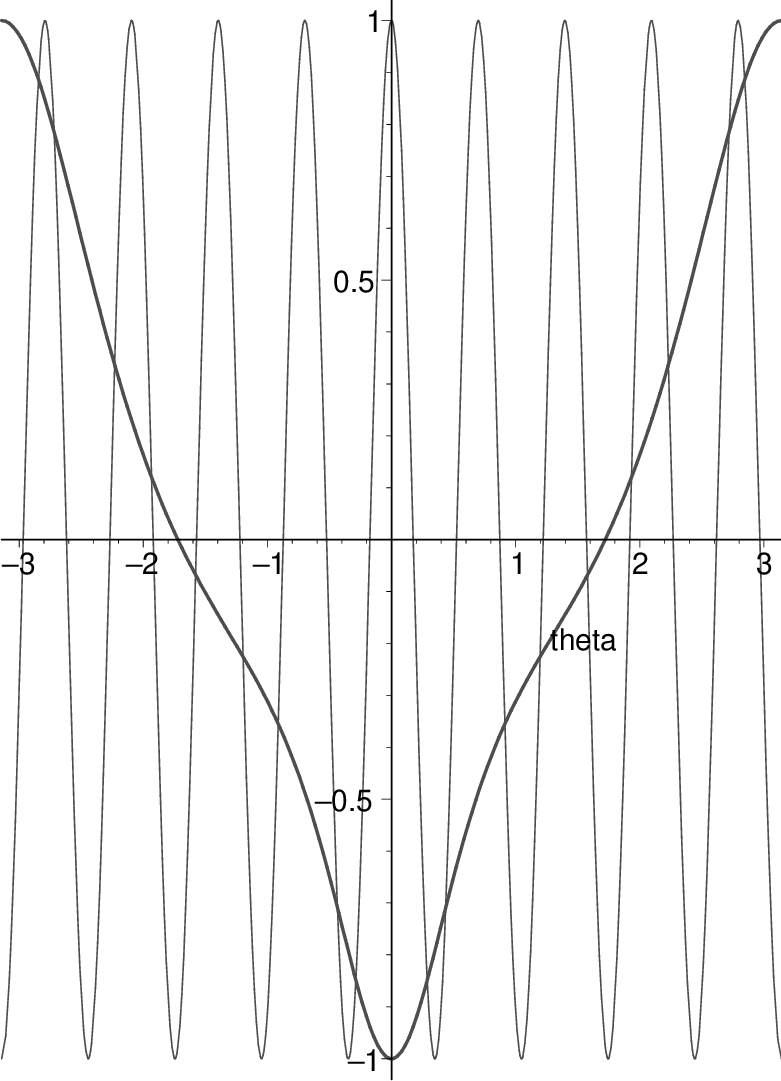}
\hskip 10mm
\includegraphics [scale=0.3]{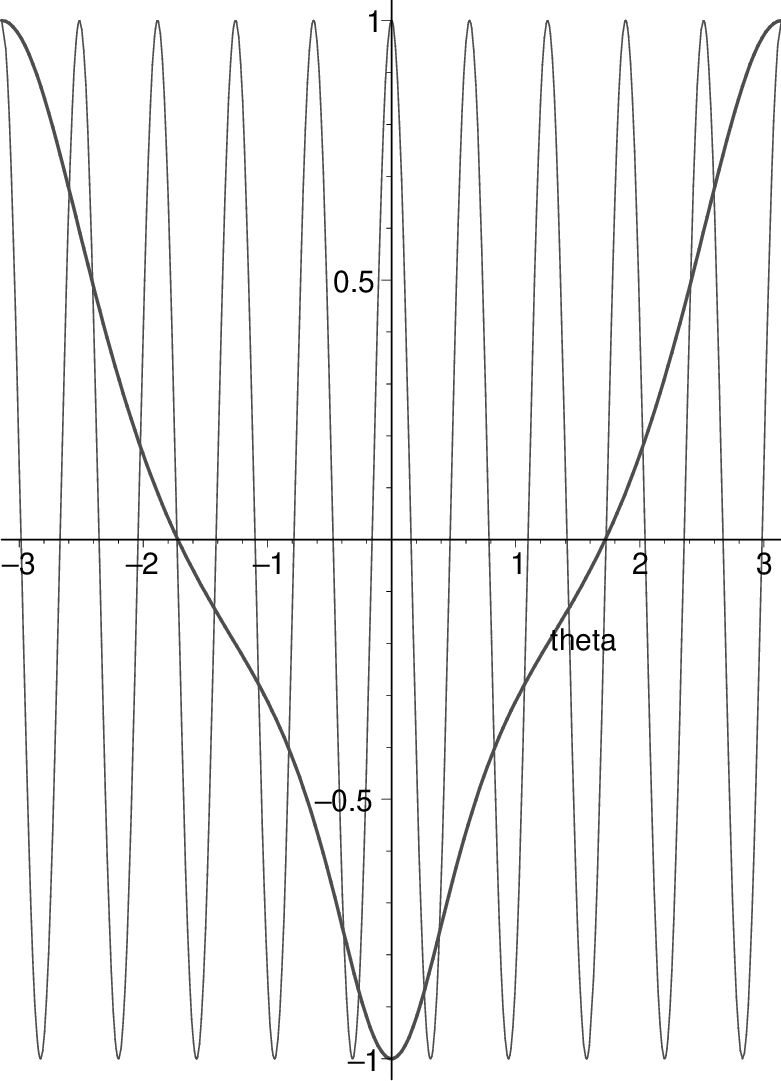}
\hskip 10mm
\includegraphics [scale=0.3]{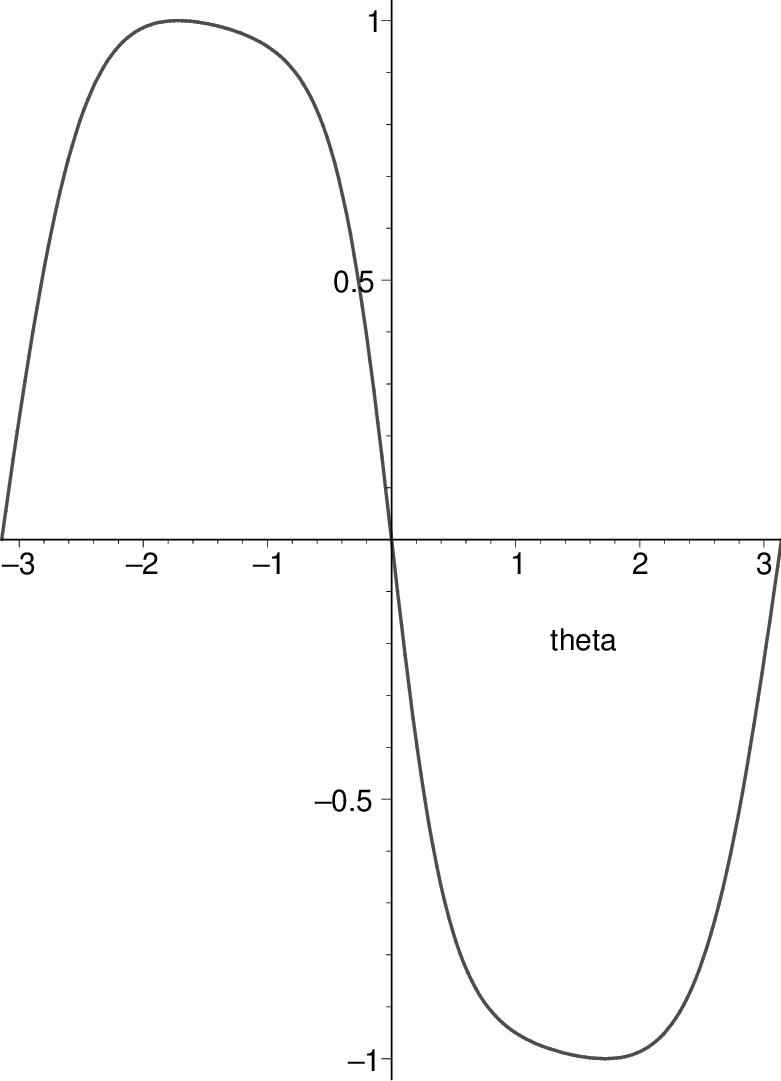}
\caption{Plots of $\cos \phi$ (thick curve) and $\cos((k-1)\theta)$
  against $\theta$, for $\theta \in [-\pi,\pi]$, when $k=10$ (left) and
  $k=11$ (middle). Right: Plot of $\sin \phi$.}
\label{fig:R}
\end{figure}

\medskip
In remains to  obtain an expansion of $\rho_k$ as $k$ grows. We first
work out an expansion of the solution of~\eqref{s-eq} found around
$1/2$ (by bootstrapping  in~\eqref{s-eq}):
$$
s= {\frac {1}{2}}-{\frac {3}{2^{k+2}}}-{\frac {3}{4^{k+2}}}
+\frac{36\,k+ 27} {4.8^{k+1}}
- \frac { {27\cdot 32}
\,{k}^{2}-9\cdot 16\,k- 717 } {16^{k+2}}
+ O\left(\frac{k^3}{32^k}\right).
$$
This translates into the expansion of $\rho_k$ using~\eqref{x-s}.
The singular behaviour of $M_k$ is then  derived from~\eqref{MS}. 
\end{proof}

\medskip
\noindent{\bf Step 4. Conclusion}\\
By Lemma~\ref{lem:zeroes} and Cauchy's formula,
\beq\label{cauchy}
[x^\ell] \left(M_k(x) - \frac{\alpha_k}{1-x/\rho_k}\right)
=
\frac 1 {2i\pi} \int_\cC\left(M_k(x) - \frac{\alpha_k}{1-x/\rho_k}\right)
\frac{dx}{x^{\ell+1}},
\eeq
where $\cC$ is the circle of radius $1/9$ centered at the origin. We
will prove that there exists a constant $C$ such that for all $k$ and
$x\in \cC$, 
$$
\left|M_k(x) - \frac{\alpha_k}{1-x/\rho_k}\right|\le Ck,
$$
so that~\eqref{cauchy} implies
$$
[x^\ell] M_k(x) = \E(X_{k,\ell} ^2)=\alpha_k\rho_k^{-\ell} + O(9^\ell k),
$$
as stated in Proposition~\ref{prop:asympt}.

It follows from~\eqref{rho-exp} and~\eqref{residue-exp} that
$\frac{\alpha_k}{1-x/\rho_k}$ is bounded uniformly in $k$ and $x\in
\cC$, so that we only need to prove that $M_k(x) =O(k)$, uniformly in
$x \in \cC$. By Lemma~\ref{lem:S}, for any $x \in \cC\setminus\{-1/9\}$,
$|S(x)|<1$. Moreover, $S(x)\rightarrow 1$ as $x \rightarrow
-1/9$. Recall that $M_k(x)= 2xN_k/G_k$. By~\eqref{G-N-S-simple}, $N_k=2^{k/2}
O(k)$, so that it suffices to prove that $G_k/2^{k/2}$ is bounded away
from 0, uniformly in $k$ and $x \in \cC$. Since $1+S$ and $S^k$ are
uniformly bounded, and $P(1/S)$ is bounded away from 0, this is
equivalent to
$$ 
\inf_{k, x\in \cC} \left|S^{k} + \frac{P(S)}{P(1/S)}\right| >0.
$$
 By the proof of Lemma~\ref{lem:zeroes}, $S^{k} + \frac{P(S)}{P(1/S)}$ does
not vanish on $\cC$. Hence it suffices to prove that
$$
\liminf_k \inf_{x\in \cC} \left|S^{k} + \frac{P(S)}{P(1/S)}\right| >0.
$$

Let us write $x=-e^{\pm i \theta}/9$, with $\theta\in [0,\pi]$. Then,
as $\theta\rightarrow 0$,
\begin{eqnarray}
  S(x)&=& 1 - \frac{1 } 2 (1\mp i) \sqrt \theta+ O(\theta), 
\label{S-exp}\\
|S(x)|&=& 1 - \frac 1  2 \sqrt \theta+ O(\theta),
\label{S-module-exp}\\
\frac{P(S)}{P(1/S)} &=& 1 - \frac{20}{13} (1\mp i)  \sqrt \theta+
O(\theta).
\label{ratio-exp}
\end{eqnarray}
We split then interval $[0,\pi]$, to which $\theta$ belongs, in three
parts.

\noindent 
$\bullet$ When $\sqrt \theta \le \pi/(2k)$, there holds, uniformly in
$\theta$,
$$
S(x)^k = \exp\left( -k(1\mp i)\sqrt \theta /2\right) + O(1/k).
$$
In particular,
\begin{eqnarray*}
  \Re( S^k) &=& \exp( -k\sqrt \theta/2) \cos (k\sqrt \theta/2) + O(1/k)\\
&\ge&  \exp(-\pi/4) /\sqrt 2+ O(1/k).
\end{eqnarray*}
Moreover,
$$
\Re \left(\frac{P(S)}{P(1/S)} \right) = 1 + O(\sqrt \theta) = 1 +
O(1/k),
$$
uniformly in $\theta$.
Hence 
$$
\Re\left( S^k +\frac{P(S)}{P(1/S)} \right) = 1+ \exp(-\pi/4) /\sqrt 2+
O(1/k),
$$
and 
$$
\liminf_k \inf_{\sqrt \theta \le \pi/(2k)} \left|S^{k} + \frac{P(S)}{P(1/S)}\right| >0.
$$

\noindent 
$\bullet$ Let $\varepsilon >0$ be such that, for $\sqrt \theta
<\varepsilon$,
$$
|S(x)| \le 1 -\sqrt \theta /4 \quad \hbox{ and } \quad
\left|\frac{P(S)}{P(1/S)} \right|>0.9.
$$
Such an $\varepsilon$ exists  in view of~\eqref{S-module-exp}
and~\eqref{ratio-exp}. For $\pi/(2k) \le \sqrt \theta \le \varepsilon$,
$$
|S|^k \le (1-\sqrt \theta /4)^k \le (1-\pi/(8k))^k=
\exp(-\pi/8)+O(1/k),
$$
so that 
$$
\left|S^{k} + \frac{P(S)}{P(1/S)}\right|\ge 0.9-\exp(-\pi/8)+O(1/k)\ge
0.2+O(1/k).
$$
Hence
$$
\liminf_k \inf_{\pi/(2k)\le \sqrt \theta \le \varepsilon} \left|S^{k} + \frac{P(S)}{P(1/S)}\right| >0.
$$

\noindent$\bullet$
Finally, when $\sqrt \theta \ge \varepsilon$, then $|S|<1$ is bounded away
from 1, uniformly in $\theta$. Thus, if 
$$
\liminf_k \inf_{ \sqrt \theta \ge \varepsilon} \left|S^{k} + \frac{P(S)}{P(1/S)}\right| =0,
$$
there would exist an $x \in \cC$ such that $P(S(x))=0$. But this only
happens when $x=0$ or $x=(-1\pm \sqrt{17})/4$, and none of these
values lies on the circle $\cC$. 

This concludes the proof of Proposition~\ref{prop:asympt}.
\end{proof}

\section{Back to Knuth's algorithm}
\label{sec:knuth}
Let us go back to Knuth's original algorithm, described at the
beginning of the paper. Recall that $\E(X_k)$ is the number of SAWs
crossing a square of side $k$, and that $\E(X_k^2)$ is the sum of the
reciprocals of the probabilities of these walks.
\begin{Proposition}\label{prop:general}
  Denote $c(k)=\E(X_k)$ and  $d(k)=\E(X_k^2)$. There exist two positive
constants $\lambda$ and $\beta$ such that 
$$
\E(X_k)^{1/k^2} \rightarrow \lambda \quad \hbox{and} \quad \E(X_k^2)^{1/k^2} \rightarrow \beta.
$$ 
Of course, $\beta\ge \lambda^2$.
Moreover,
\beq\label{lambda-bound}
\lambda= \sup_k\  c(k)^{1/(k+1)^2}
\eeq
and
\beq\label{beta-bound}
\beta= \sup_k\  (\sqrt 2\, d(k))^{1/(k+1)^2}.
\eeq
\end{Proposition}
As discussed at the end of the introduction, there is a hope to
combine~\eqref{beta-bound} and known upper bounds on $\la$ to prove
that $\beta >\lambda ^2$, in which case one could conclude that the
relative variance of $X_k$ grows as $\kappa^{k^2}$, with
$\kappa=\beta/\lambda^2>1$. 
\begin{proof}
  As can be expected, these results follow from a
  super-multiplicativity argument. The existence of $\lambda$ was
  established for the first time in~\cite{abbott},
  and~\eqref{lambda-bound} (which allows to produce lower bounds on
  $\la$) appears in~\cite{mbm-crossing}. We repeat the argument, 
  because it applies almost verbatim to the numbers $d(k)$.

Define $\lambda:= \limsup_k c(k)^{1/k ^2}$. Then $\lambda$ is finite,
because there are only a quadratic number of edges in the $k\times k$
square, and a walk is determined by the set of its edges. Let
$\varepsilon >0$. We will prove that
\beq\label{liminf}
\liminf c(K)^{1/K^2}\ge \lambda-\varepsilon,
\eeq
which implies that $\lambda$ is actually the limit of $c(k)^{1/k^2}$.

Let $k>0$ be such that $c(k)^{1/(k+1)^2}>\lambda-\varepsilon$. Let
$K\ge k$,
and let $n$ be maximal so that
$$
(k+1)(2n+1)-1 \le K.
$$
This implies in particular that $K<(k+1)(2n+3)$.
In the $K\times K$ square, put $(2n+1)^2$ smaller squares of side $k$,
as shown in Figure~\ref{fig:mult}. In each smaller square, choose a SAW
that crosses it, and build from this collection of short walks a long walk
crossing the larger square, as shown in the figure. This construction implies
$$
c(K) \ge c(k)^{(2n+1)^2}.
$$
Thus
$$
c(K)^{1/K^2} \ge c(K)^{1/((k+1)^2(2n+3)^2)} \ge
\left(c(k)^{1/(k+1)^2}\right)
^{(2n+1)^2/(2n+3)^2}\ge (\lambda-\varepsilon)^{(2n+1)^2/(2n+3)^2}.
$$
 Taking the $\liminf$ on $K$ boils down to taking the $\liminf$ on $n$ and
 gives~\eqref{liminf}. 
The bound~\eqref{lambda-bound} also follows from the above inequalities.

\begin{figure}[htb]
  \begin{center}
   \scalebox{0.8}{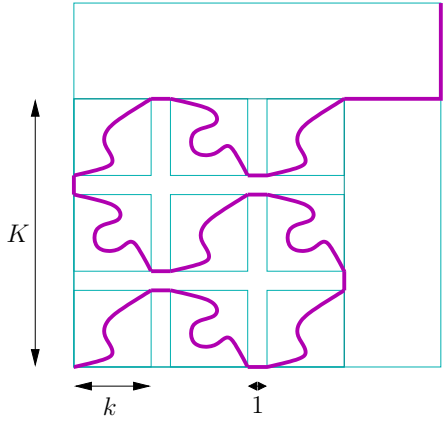}
  \end{center}
  \caption{Super-multiplicativity for SAWs crossing a square}
\label{fig:mult}
  \end{figure}

\medskip
Let us now consider the numbers $d(k)$. Again, $\beta:= \limsup
d(k)^{1/k^2}$ is finite, because
$$
d(k)=\sum_{w\in \cW_k} \frac 1 {p(w)} \le \sum_{w\in \cW_k} 3^{|w|} \le
 3 ^{O(k^2)}c(k),
$$
where $|w|$ denotes the length of $w$.
Now return to Figure~\ref{fig:mult}.  Denote by $w_1, w_2, \ldots$ the
short walks, and by $w$ the long one. It is clear for the sampling
algorithm that
$$
\frac 1 {p(w)} \ge \prod_{i=1}^{(2n+1)^2} \frac 1 {p(w_i)}.
$$
It follows that 
$$
d(K) \ge d(k)^{(2n+1)^2},
$$
from which one can prove, as above, that $\beta= \lim d(k)^{1/k^2}$. The above
bound on $d(K)$ can actually be improved: in every row of $(2n+1)$ small
squares, except maybe the top one, $n$ of the horizontal steps added
between the small squares have probability $1/2$ or $1/3$. Hence
$$
d(K) \ge 2^{2n^2}d(k)^{(2n+1)^2},
$$
and the lower bound~\eqref{beta-bound} now follows.
\end{proof}

\section{Final comments}
\subsection{Unconfined walks}
\label{sec:unconfined}
Knuth designed his algorithm to sample SAWs crossing a square of side $k$, but
other authors have used similar ideas to sample unconfined SAWs of fixed
 length $n$. For example, the classical Rosenbluth algorithm~\cite{rosenbluth} generates
 general SAWs step by step, by taking at each time, uniformly at random, one of
 the steps that preserves self-avoidance. If at some point no such
 step is available, and the walk has not reached length $n$, the algorithm restarts
 from scratch. This rejection step is avoided if
 one only samples \emm untrapped, self-avoiding walks, that is, 
walks that can be extended into a SAW of infinite length\footnote{This
  notion of untrapped walks differs from the one
  in~\cite{aleks-thomas}, where a walk is said to be untrapped as
soon as it can be extended by one step.}. (We describe
 in Section~\ref{sec:trap} a simple procedure that detects if a
new step traps the walk, which is also useful for implementing Knuth's
algorithm.)  A
 recent numerical study, using a refinement of the above algorithm,
 suggests that the asymptotic properties  of untrapped walks are
 similar to those of general SAWs, in terms of number and end-to-end
 distance~\cite{chan-rechni}. 

For these algorithms, the quality of the cardinality estimator is still related to
the variance of the random variable $X_n$ equal to the reciprocal of
the probability of the generated walk.
As already mentioned, the variance of the Rosenbluth estimator is
predicted to be exponential in $n$~\cite{batoulis-kremer}. We do not
know of any similar study for untrapped walks. It is easy to determine
the variance of $X_n$ for the Rosenbluth algorithm restricted to directed or partially directed walks. Our
results are summarized in the following table. In particular, for
partially directed walks the relative variance is found to be exponential in
$n$.

\bigskip
\begin{tabular}{c|lc|c|c}
&& North and East & North, East and South & all four steps \\
\hline
confined &number&   $4^ k/\sqrt k$  & $(k+1)^k$ & $\lambda^{k^2}$ \\
to $k\times k$ &rel. var.& $\sqrt k$ \ \ \cite{bassetti-diaconis} & $2^{k(k+1)}/(k+1)^k $ (Prop.~\ref{prop:asympt}) &
$\kappa^{k^2}$ ? (Prop.~\ref{prop:general}) \\
 &av. length&  $k$ & $k^2$ & $k^2$\ \  \cite{madras}\\
\hline
&&&&\\
unconfined, & number&$2^n$ &  $(1+\sqrt 2)^n$ & $(2.64...)^n$ \cite{madras-slade}\\
$n$ steps &rel. var. &0 & $ (6/(1+\sqrt 2))^n$ & $\alpha^n$ (pred.~\cite{batoulis-kremer})
\end{tabular}

\medskip

\subsection{Kinetic distributions}
It is also interesting to study the asymptotic properties of
 SAWs chosen according to the non-uniform (but very natural)
``kinetic'' distribution that results from importance sampling. These
 properties may be  different from those observed  in the
 uniform case. For instance, 
one can expect the average end-to-end 
distance of kinetic unconfined SAWs to be smaller than $n^{3/4}$, because ``compact'' walks in which few
steps are eligible at each time have a higher probability than more
spread-out walks. In fact, the kinetic end-to-end distance is conjectured~\cite{majid} to
grow like $n^{2/3}$. (For unconfined partially directed
walks, however, the end-to-end distance is easily shown to be linear,
both for the uniform and the kinetic model.) Figure~\ref{fig:unconfined} shows a random (untrapped)
SAW that we generated by importance sampling and a (quasi-)uniform SAW
generated using a pivot algorithm~\cite{kennedy-pivot-saw}. 

\begin{figure}[htb]
\includegraphics[scale=0.3]{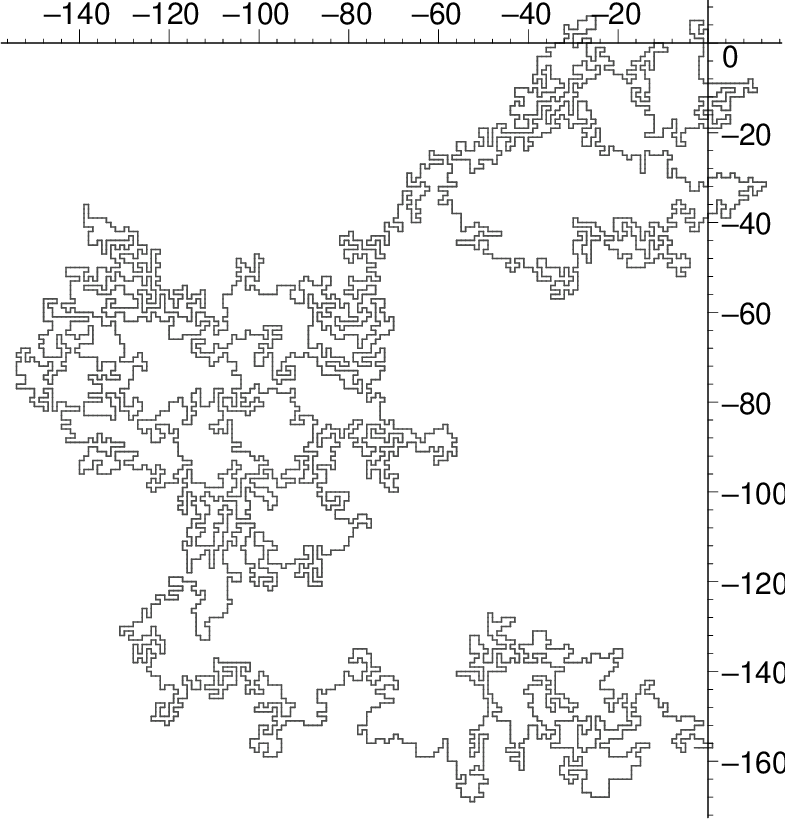} 
\hskip 20mm \includegraphics[height=4.5cm,width=3cm]{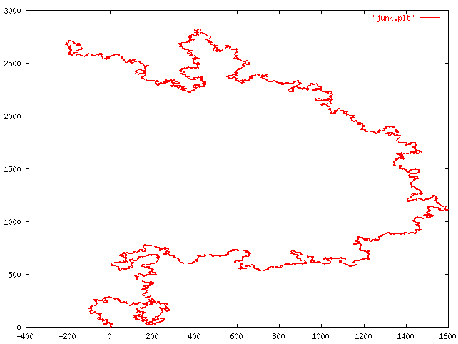} 
\caption{A random untrapped SAW of length 5000 obtained via importance
  sampling (left), and a quasi-uniform SAW of length 20000 (right).}
\label{fig:unconfined}
\end{figure}

\subsection{When does a walk get trapped?}
\label{sec:trap}
One important feature of Knuth's algorithm, and of its adaptation to
untrapped SAWs discussed in Section~\ref{sec:unconfined}, is that one never
appends a step that would trap the walk. Since Knuth does not explain
in his paper how he detects trapping, let us describe the method
we used. One obvious case of trapping in Knuth's algorithm is when
the walk reaches  the boundary of the square, and moves towards the
origin. In all other trapping situations, the walk would have been
trapped as well in the unconfined setting, so we focus on the trapping of
unconfined walks. 

Let $w$ be an untrapped SAW of length $n$, ending at vertex $v_n=(i,j)$, and,
say, with a \WW\ step. There are, up to obvious symmetries, exactly three situations when adding
a new step to $w$ creates a trapped walk:
\begin{itemize}
\item the vertex $v=(i-1,j)$ belongs to $w$, one appends a \NN\ step
  to $w$ and the portion of $w$ going from $v$ to $v_n$ has \emm
  winding number,  $-2\pi$,
\item the vertex $v=(i-1,j+1)$ belongs to $w$, one appends a \NN\ step
  to $w$ and the portion of $w$ going from $v$ to $v_n$ has 
  winding number  $-2\pi$,
\item the vertex $v=(i-1,j+1)$ belongs to $w$, one appends a \WW\ or
  \SS\ step
  to $w$ and the portion of $w$ going from $v$ to $v_n$ has 
  winding number  $2\pi$.
\end{itemize}

\begin{figure}[htb]
\includegraphics[scale=0.8]{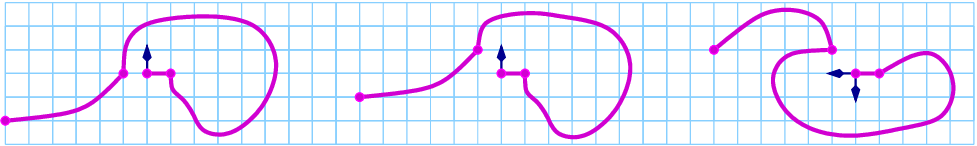} 
\vskip -3mm\caption{How a walk gets trapped.}
\label{fig:trapped}
\end{figure}

These three cases are depicted in Figure~\ref{fig:trapped}. When computing the
winding number, we add a half-edge pointing from the East to $v$ (Figure~\ref{fig:winding}). The
winding number is then the difference between the number of left turns
and the number of right turns, multiplied by $\pi/2$.

\begin{figure}[htb]
  \begin{center}
   \scalebox{0.6}{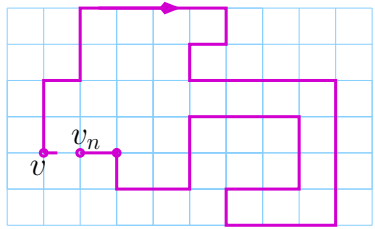}
  \end{center}
\vskip -3mm  \caption{The winding number between $v$ and $v_n$ is $-2\pi$.}
\label{fig:winding}
  \end{figure}

\medskip
\noindent
{\bf Acknowledgements.} I am very grateful to Persi Diaconis, who
first told me about this question, and then  sent helpful
suggestions on this manuscript and pointed to interesting related
papers.
I also thank the referee of a first version of this paper for his very
thorough report and helpful references.
\bibliographystyle{plain}
\bibliography{sampling.bib}

\end{document}